\documentclass{amsart}

\usepackage{color}
\usepackage{amsxtra}
\usepackage{mathtools}
\usepackage{enumerate}
\usepackage{blindtext}
\usepackage{nicefrac}
\usepackage{color}

\usepackage{hyperref}
\hypersetup{
  colorlinks   = true, 
  urlcolor     = blue, 
  linkcolor    = blue, 
  citecolor   = red 
}
\newtheorem{theorem}{Theorem}[section]
\newtheorem{lemma}{Lemma}[section]
\newtheorem{proposition}{Proposition}[section]

\newtheorem{remark}{Remark}[section]

\title[The evolution problem for the fractional first eigenvalue]{The evolution problem associated with the fractional first eigenvalue}
\author[B. Barrios]{Bego\~na Barrios}
	\address{Bego\~na Barrios \hfill\break\indent
		Departamento de An\'{a}lisis Matem\'{a}tico,
		Universidad de La Laguna\hfill 
		\break \indent C/. Astrof\'{\i}sico Francisco S\'{a}nchez s/n, 
		38200 -- La Laguna, SPAIN}
	\email{bbarrios@ull.es}

\author[L. M. Del Pezzo]{Leandro M. Del Pezzo}
	\address{Leandro M. Del Pezzo and Julio D. Rossi \hfill\break\indent
        Departamento  de Matem{\'a}tica, FCEyN,\hfill\break\indent
        Universidad de Buenos Aires,
        \hfill\break\indent Pabellon I,
         Ciudad Universitaria (C1428BCW),
         \hfill\break\indent 
        Buenos Aires, Argentina.}
\email{ldpezzo@dm.uba.ar, jrossi@dm.uba.ar}

\author[A. Quaas]{Alexander Quaas}
\address{A. Quaas
       \hfill\break\indent Departamento de Matem\'atica, 
        \hfill\break\indent Universidad T\'ecnica Federico Santa Mar\'ia 
        \hfill\break\indent Casilla V-110, Avda. 
        \hfill\break\indent Espa\~na, 1680 -- Valpara\'iso, CHILE.}
    \email{alexander.quaas@usm.cl}
    
\author[J. D. Rossi]{Julio D. Rossi}

\begin{document}

\begin{abstract} In this paper we study the evolution problem
associated with the first fractional eigenvalue. We prove that 
the Dirichlet problem with homogeneous boundary condition is
well posed for this operator in the framework of viscosity solutions
(the problem has existence and uniqueness of a solution and
a comparison principle holds). In addition, we show that solutions
decay to zero exponentially fast as $t\to \infty$ with a bound that is
given by the first eigenvalue for this problem that we also study.

\smallskip


\smallskip

\noindent \textbf{2020 AMS Subj. Class.} {35K55, 35D40, 35R11} 
\end{abstract}

\maketitle




\section{Introduction}
    The main goal of this work is to study the evolution problem
	associated to the following non-linear non-local operator 
	\begin{equation}\label{operador}
	    \Lambda_1^s u(x)\coloneqq\inf_{\theta\in\mathbb{S}^{N-1}}\int_{\mathbb{R}}
        \dfrac{{u}(x+\tau \theta)-{u}(x)}{|\tau|^{1+2s}} d\tau,
    \end{equation}
	where $s\in(0,1),$ and  $\mathbb{S}^{N-1}$ is the $(N-1)-$sphere. Here and throughout the entire 
	manuscript, the integrals are understood  in the principal value sense.
	
	Notice that in the previous definition the infimum is computed among directions $\theta$ of the quantities given by the $1-$dimensional
	fractional Laplacian for $u$ restricted to the line with direction $\theta$ that passes
	through the point $x$.
	
	 Here we will consider $N>2$. Observe that, trivially, when $N=1$ the operator given 
	 in \eqref{operador} coincides with $-(-\Delta)^{s}_{1}$, the negative value of the $1-$dimensional 
	 fractional Laplacian for which the results that we obtain in the present work are known. 

    This paper is a natural continuation of \cite{DPQR_FC} where the elliptic problem
	$\Lambda_1^s u(x)=0$ inside a convex domain with a Dirichlet exterior condition
	is studied. However, we want to highlight that in the present work we establish some relevant 
	results also related with the elliptic problem like the regularity up to the boundary.
	
	\medskip

	The operator described in \eqref{operador} appears naturally when we extend the notion
	of convexity of a function to the non-local setting. Let us describe this 
	fact with more detail. 
	It is well
	known that a function $u\colon\Omega\to\mathbb{R}$ is said to be 
	convex in $\Omega$
	if for any two points $x,y\in\Omega$ such that the segment that goes from $x$ to $y$ is contained in $\Omega$, i.e.,  
	\[
		[x,y]\coloneqq\{tx+(1-t)y\colon t\in(0,1)\}\subseteq\Omega,
	\] 
	it holds that 
	\[
		u(tx+(1-t)y)\le t u(x)+(1-t) u(y), \quad t\in(0,1).
	\]
	See for instance, \cite{Vel}.
	In terms of  second order partial differential equations, a function
	$u$ is convex if and only if
	\[
		\lambda_1(D^2u)(x)\ge0,
	\]
	in the viscosity sense. Here $\lambda_1(D^2u)$ denotes the 
	smallest eigenvalue of the Hessian, $D^2u,$ that is
	\begin{equation}\label{eq_intro_lamba1}
		\lambda_1(D^2u)(x)\coloneqq\inf
		\left\{	
			\langle D^2u(x) z,z\rangle\colon z\in\mathbb{S}^{N-1}
		\right\}.
	\end{equation} 
	We refer to \cite{BlancRossi,HL1,OS,Ober}, and references therein.
	
	In \cite{DPQR_FC}, inspired by the classical notion of convexity,
	 the authors propose the following natural extension 	
	of convexity for functions in the non-local setting: a function
	$u\colon \mathbb{R}^N\to \mathbb{R}$ is said $s-$convex in $\Omega$ if
	for any $x,y\in\Omega$ such that $[x,y]\subseteq\Omega$ it holds that
	\[
		u(tx+(1-t)y)\le v(tx+(1-t)y),\quad t\in(0,1),	
	\] 
	where now $v$ is the viscosity solution of the $1-$dimensional fractional 
	Laplacian of order $s$ in the line $t \mapsto tx + (1-t)y$ for $t \in (0,1)$ with
	exterior values $u(tx+(1-t)y)$ for $t \not\in (0,1)$, that is, $v$ is the solution to 
		\begin{align*}
		&\Delta_1^sv(tx+(1-t)y)
			\coloneqq \int_{\mathbb{R}}
			\dfrac{{v}(\tau x+(1-\tau)y)-{v}(tx+(1-t)y)}{|t-\tau|^{1+2s}} d\tau=0 
			\quad t\in(0,1),\\
		&{v(tx+(1-t)y)}=u(tx+(1-t)y),\quad t\not \in(0,1).
	\end{align*}
	As usual, the integral is to be understood in the principal
	value sense. Finally, in \cite{DPQR_FC}, it is shown that $u$
	is $s-$convex in a bounded strictly convex $\mathcal{C}^{2}$ domain $\Omega$
	(from this point, $\Omega$ will be 
	a bounded domain with $\mathcal{C}^2$-boundary such that all the principal curvatures of the 
	surface $\partial\Omega$ are positive everywhere) if only if 
	it satisfies
	\[
		\Lambda_1^s u(x)\ge 0\text{ in } \Omega,
	\] 
	in the viscosity sense where $\Lambda_1^s$ is the operator given by \eqref{operador}.
	
    We notice that this operator \eqref{operador} also appears as part of a family of operators
	called the truncated fractional Laplacians that where studied in \cite{TBG} where the authors proved that $C(s)\Lambda_1^s u(x)\to \lambda_1(D^2u)(x)$ with $C(s)\to 0$ when $s\to 1^{-}$.
	For the local counterpart of truncated Laplacians we refer to \cite{BGI,BGI2,BGL}.  
	As far as we know this is the first reference dealing with the associated evolution problem. 
	
	\medskip
	
	Coming back to the main objective of the work, we will study the following problem	
	\[
		\begin{cases}
			u_t(x,t) =\Lambda_1^s u(x,t) &\text{in }\Omega\times(0,\infty),\\
			u(x,t)=g(x) &\text{in }
				(\mathbb{R}^N\setminus\Omega)\times (0,\infty),\\
			u(x,0)=u_0(x) &\text{in }{ \mathbb{R}^N},
		\end{cases}
	\]
	where $\Omega$ is a bounded and strictly convex domain, 
	$g$ is a function defined in $\mathbb{R}^N\setminus\Omega$ and $u_0$ is a fixed initial condition 
	defined in $\mathbb{R}^N$. 
	
	Since the operator \eqref{operador} is not in divergence form 
	the natural framework to treat this problem is to use the theory 
	of viscosity solutions (see Section \ref{sect-evolution} for more details). For this notion of solutions 
	we refer to the celebrated work of \cite{CIL} for the classical local concept and, for instance, to 
	\cite{CS, RO} for a non-local counterpart. It is good to bear in mind that, in order to guarantee the 
	attainability of the datum in the notion of the viscosity solution, we have to impose some reasonable 
	conditions. Indeed in what follows we will assume that 
	\begin{equation}\label{hip}	
			\Lambda_1^su_0(x)\in L^{\infty}(\Omega)  ,\, 
			g\in \mathcal{C}({\mathbb{R}^N\setminus\Omega})\cap 
		L^\infty({\mathbb{R} ^N\setminus\Omega})
		\mbox{ and } 
			u_0=g \mbox{ in }	\mathbb{R}^N\setminus\Omega.
	\end{equation}
    
    \medskip
	
	Our first result, presented in the following theorem, 
	that says  the evolution problem is well posed, that is the problem has existence and uniqueness of a 
	solution and a comparison principle holds. 
	
	\begin{theorem} \label{teo-1.intro} 
	If  $\Omega$ is a bounded strictly convex domain,
	$u_0{\in\mathcal{C}(\mathbb{R}^N)}$ and $g$ verifying \eqref{hip}, then there exists a unique viscosity solution to 
	\begin{equation}\label{ev_eq}
		\begin{cases}
			u_t(x,t)=\Lambda_1^s u(x,t) &\text{in }\Omega\times(0,\infty),\\
			u(x,t)=g(x) &\text{in }
				(\mathbb{R}^N\setminus\Omega)\times(0,\infty),\\
			u(x,0)=u_0(x) &\text{in } \mathbb{R}^N. 
		\end{cases}
	\end{equation}
	Moreover, a comparison principle holds; that is, if $\overline{u}$ is a supersolution 
	 and $\underline{u}$ is a subsolution to \eqref{ev_eq} then it holds that
	 $$\overline{u} (x,t) \geq \underline{u} (x,t),$$ for every $(x,t) \in \Omega \times (0,\infty)$. 	
	\end{theorem}

	Our next objective is to study the asymptotic behavior of the solution as $t\to \infty$
	in the special case $g \equiv 0$. 
	When we deal with the homogeneous case, $g \equiv 0$, we can obtain an 
	exponential bound for the rate of convergence to zero. 
	In parabolic problems the exponential decay rate
	is associated with the first eigenvalue of the operator.
	In our case this eigenvalue problem reads as
		\[
			\begin{cases}
				-\Lambda_1^s \varphi_1(x)=\mu_{1}\varphi_1 
				&\text{in }\Omega,\quad\\
			 \varphi_1=0&\text{in } \mathbb{R}^{N}\setminus\Omega.
				\end{cases}
		\]
    At this point we mention that this eigenvalue problem and a regularity issue associated to the 
    operator \eqref{operador} were also studied in the recent work \cite{biri} but we include the result 
    here because our proof is different. To prove the existence of the eigenfunctions we need a regularity 
    result that we consider interesting in its own right. 
    For that, we need the extra assumption $s>\nicefrac12$ related to the necessity of the positivity of the exponent of the
    fundamental solution of the Fractional Laplacian in one-dimension. That is we have the following
    \begin{theorem}\label{regularidad.intro} 
	    Let $u$ be a viscosity solution of 
        $$
            \begin{cases}
		        -\Lambda_1^s u(x)=f(x) & \mbox{in}\quad \Omega,\\
		        u(x)=g(x) & \mbox{in}\quad \mathbb{R}^N\setminus\Omega.
		        \end{cases}
        $$
        Assume that  $s>\nicefrac12,$ $f$ is a bounded function and $g$ 
        satisfies a H\"older bound, so that there exist $M_g$ 
        and $\beta\in(s,2s)$    such that
        \begin{equation}\label{holder-g}
                |g(x)-g(y))| \leq M_g |x-y|^{\beta},\, x,\, y
                \in \mathbb{R}^N\setminus\Omega.
        \end{equation}
        Then,  there exists 
        $C=C(\|f\|_{L^\infty(\Omega)},M_g,\|u\|_{L^\infty(\Omega)},    \Omega)>0$ 
        such that
     	\[
     			\|u\|_{\mathcal{C}^\gamma({\bar\Omega)}}\leq C,
 		\]
 		where 
 		\[
 		    \gamma\in 
 		    \begin{cases}
	            (0,2s-1) &\text{ if } g\equiv0,\\
	            (0,\beta-s)&\text{ if } g\neq0.\\
        \end{cases}
 		\]
    \end{theorem}

 Then, we obtain the following result
for the asymptotic behavior of solutions to our evolution equation (first we show exponential convergence to the steady state and, in the
special case of $g \equiv 0$ and $u_0$ compactly supported inside
$\Omega$ we obtain more precise bounds that imply that the exponential rate is given by the first eigenvalue).

\begin{theorem} \label{teo.conver.intro} Let $s>\nicefrac12$.
If $u_0$ and $g$ satisfy \eqref{hip} 
		then there exist $C(u_0)>0$ and $\mu>0$ such that 
		\begin{equation}\label{c_uniforme}
			 \|  u(\cdot,t)-z(\cdot) \|_{L^\infty (\overline\Omega)} 
			\leq C e^{-\mu t},\quad t>0,
		\end{equation}
		where $u(x,t)$ is the unique viscosity solution of \eqref{ev_eq} 
		and $z$  {is the $s$-convex evelope, that is,} the unique viscosity solution of 
		\[
			\begin{cases}
				-\Lambda_1^s z(x)=0&\text{in }\Omega,\\
				z=g&\text{in } \mathbb{R}^{N}\setminus\Omega.
			\end{cases}
		\]
		
		When $g \equiv 0$ and $u_0$ is compactly supported inside $\Omega$ 
		we have that, {\color{black} for every $\mu > 0$ there exist $C{>0}$, $K=K(\mu)$
		and $H{>0}$ an universal constant such that
		\[
		- C e^{-\mu_1 t}	\leq u(x,t)\leq  K e^{-\mu^s Ht},\quad x\in\overline{\Omega},\quad t>0.
		\]
		with $u(x,t)$ the unique viscosity solution of \eqref{ev_eq} given 
		by Theorem \ref{teo-1.intro}. 
		Here $\mu_1$ denotes the first eigenvalue
		of the operator. 
}		
\end{theorem}


For similar results established in the local case for the evolution problem associated with 
	\eqref{eq_intro_lamba1} we refer to \cite{BlancRossi}. 
	
		 \begin{remark}
		{\rm Some of the results obtained 
		in the present work could be extended to other kind of operators like 
		\[
	        \Lambda_j^s u(x)\coloneqq \sup_{S } 
	        \inf_{\theta\in S \cap \mathbb{S}^{N-1}}\int_{\mathbb{R}}
            \dfrac{{u}(x+\tau \theta)-{u}(x)}{|\tau|^{1+2s}} d\tau.
            \]
            Here $S$ is a subspace of dimension $N-j+1$ in $\mathbb{R}^{N}$.
            This non trivial extension could be obtained as soon as one proves a comparison
            principle for this kind of operators. }
	\end{remark}

	\medskip
	
	The paper is organized as follows: in Section \ref{sect-regul}
	we deal with the regularity result and prove Theorem \ref{regularidad.intro}; 
	in Section \ref{sect-eigen} we prove the existence of the first eigenvalue (with
	a corresponding eigenfunction); finally, in Section \ref{sect-evolution} 
	is devoted to the study of
	the evolution problem, that is, we prove Theorem \ref{teo-1.intro} and 
	Theorem \ref{teo.conver.intro}. 
	
	\medskip
	
	We remark here that along the whole work we will denote by $C$ a positive constant that 
	may change from line to line and that we use $B_R(x)$ to denote the ball of radius 
	$R>0$ centered at $x\in\mathbb{R}^{N}.$

\section{Regularity, up to the boundary, for the elliptic problem} \label{sect-regul}
    We will assume that $s>\nicefrac12$ in the whole section and the main 
    objective will be the study of the H\"older regularity, up to the boundary, of a viscosity solution to
    \begin{equation}\label{eq:regularidad}
        \begin{cases}
		-\Lambda_1^s u(x)=f(x) & \mbox{in}\quad \Omega,\, s>\nicefrac12\\
		u(x)=g(x) & \mbox{in}\quad \mathbb{R}^N\setminus\Omega,
		\end{cases}
    \end{equation}
	where $f$ is a bounded function and $g\in\mathcal{C}^{\beta}$ for some $s<\beta<2s$. 
	First of all, it is worth to bear in mind that the notion of viscosity solution that we will use for the 
	elliptic problem can be found in \cite{DPQR_FC}. We start by introducing some notation. For every 
	$x\in B_1(0)$ and $\theta\in \mathbb{S}^{N-1},$ we
    define
    \[
	  d(x)\coloneqq (1-|x|^2)_+, \quad \mu(x,\theta)\coloneqq\frac{\langle x,\theta \rangle}{|x|}
	  \in[-1,1],
    \]
    and
    \[
         p_{x,\theta}(t)\coloneqq t^2+2\mu(x,\theta)|x|t-d(x).
    \]
    We denote by $t_0(x,\theta)$ and $t_1(x,\theta)$ the two roots of the $p_{x,\theta}(t),$ that is,
    \begin{align*}
	&t_0(x,\theta)\coloneqq-\langle x,\theta \rangle-\sqrt{d(x)+\langle x,\theta \rangle^2}< 0,\\
	&t_1(x,\theta)\coloneqq-\langle x,\theta \rangle+\sqrt{d(x)+\langle x,\theta \rangle^2}>0,
    \end{align*}
    that satisfy
    \begin{equation}\label{tt01}
        t_0(x,\theta)t_1(x,\theta) =-d(x).
    \end{equation}
    
    Let us define our first barrier by	
    \[
	w_\gamma(x)\coloneqq d(x)^\gamma, \quad \gamma\in\mathbb{R},
    \]
    and the integral
    \[
        \mathcal{I}_{\theta} (f)(x)\coloneqq\int_{\mathbb{R}}
			\dfrac{f(x+\tau \theta)-f(x)}{|\tau|^{1+2s}} d\tau,\quad \theta\in\mathbb{S}^{N-1}.    
    \]
    
    Taking into account the previous notations, we can formulate our first lemma.
    
    \begin{lemma} \label{barrera_2}
	If $0<\gamma<2s-1$ then there exits $c_1>0$ such that
	\[
	    \mathcal{I}_{\theta} (w_\gamma)(x)\leq w_\gamma(x)  \left[
	    	-\frac{t_1(x,\theta)^{-2s} + |t_0(x,\theta)|^{-2s}}{2s}
	    -c_1\left(\frac{2 {|\mu(x,\theta)|}|x|}{d(x)}\right)^{2s}\right],
	\]
	for all $ x\in B_1(0)$ and $\theta\in \mathbb{S}^{N-1}.$ In particular, 
	there exists $c_2>0$ such that
	\[
	    \mathcal{I}_{\theta}(w_\gamma(x))\leq -c_2 d(x)^{\gamma- {s}},
	    \quad x\in B_1(0),\, \theta\in \mathbb{S}^{N-1}.\]
    \end{lemma}
    \begin{proof}
        Let $ x\in B_1(0)$ and $\theta\in \mathbb{S}^{N-1}.$
        To simplify the notation, let us now temporarily write  $w=w_\gamma,$
	    $\mu=\mu(x,\theta),$ and $t_i=t_i(x,\theta)$ with $i=0,1.$ First of all we notice that
	\begin{equation}
        \label{naux1}
        \begin{aligned}
	        \mathcal{I}_{\theta} (w)(x)&=\int_{\mathbb{R}}
	        \dfrac{(1-|x+\tau\theta|^2)^\gamma_+-(1-|x|^2)^\gamma}{|\tau|^{1+2s}}d\tau\\
		    &=\int_{\mathbb{R}}
		    \dfrac{(1-|x|^2-2\mu |x|\tau-\tau^2)^\gamma_+-(1-|x|^2)^\gamma}{|\tau|^{1+2s}} 
		    d\tau\\
		    &=\int_{\mathbb{R}\setminus(t_0,t_1)}
		    \dfrac{(1-|x|^2-2\mu |x|\tau-\tau^2)^\gamma_+-(1-|x|^2)^\gamma}{|\tau|^{1+2s}} 
		    d\tau+ J,
	    \end{aligned}
    \end{equation}
	where
	\[
	    J\coloneqq\int_{t_0}^{t_1}
			\dfrac{(1-|x|^2-2\mu |x|\tau-\tau^2)^\gamma-(1-|x|^2)^\gamma}{|\tau|^{1+2s}} d\tau.
	\]
	Since by definition of $t_i$, $i=0,1$, it is clear that
	\[
	    1-|x|^2-2\mu |x|\tau-\tau^2< 0, \quad\tau\in (-\infty,t_0)\cup(t_1,\infty),\, 
	     {\mu\in[-1,1]},
	\]
    from \eqref{naux1} we get
	\[
	    \mathcal{I}_{\theta} (w)(x)\le -\frac{w(x)}{2s}\left( t_1^{-2s}  + |t_0|^{-2s}\right)+J.
	\]
	So, to conclude, we just need to estimate $J$. For that, by symmetry, we split the rest of the proof in 
	the two cases $\mu>0$ and $\mu=0.$
    \medskip

	\noindent {\it Case 1: $\mu>0.$} Since it is clear that
	\[
	    J\leq\int_{t_0}^{t_1}
				\dfrac{(1-|x|^2-2\mu |x|\tau)^\gamma-(1-|x|^2)^\gamma}{|\tau|^{1+2s}} d\tau
				= w(x) \int_{t_0}^{t_1}
				\dfrac{\left(1-\tfrac{2\mu |x|}{d(x)}\tau\right)^\gamma-1}{|\tau|^{1+2s}} d\tau,
	\]
	by making the change of variable $\upsilon=\frac{2\mu |x|}{d(x)}\tau$ we get
	\[
	    J\leq w(x) \left(\frac{2\mu|x|}{d(x)}\right)^{2s}
			\int_{\upsilon_0}^{\upsilon_1}\frac{(1-\upsilon)^\gamma-1}{|\upsilon|^{1+2s}}
			d\upsilon=:w(x) \left(\frac{2\mu|x|}{d(x)}\right)^{2s}J_1,
	\]
	where $\upsilon_i=\tfrac{2\mu |x|}{d(x)} t_i=1-\tfrac{t^2_i}{d(x)}$, $i=0,1.$ Proceeding now as 
	in \cite[Lemma B.1]{BCGJ}, by making the change of variable  $e^\rho =1-\upsilon$, only in  
	$\rho \in [\rho_1(x), \rho_0(x)]$ (not in all the Euclidean space as in \cite{BCGJ}) where
	$\rho_i(x)=\ln \left(\tfrac{t_i^2}{d(x)}\right)$, $i=0,1$, we obtain that 
	\[
	    J_1=2^{-2s}\int_{\rho_1}^{\rho_0}e^{(\gamma+1-2s)\frac{\rho}2}
				\sinh\left(\gamma \frac{\rho}2\right)
				\left|\sinh\left(\frac\rho2\right)\right|^{-(1+2s)}d\rho
				\begin{cases}
					<0 &\text{if } \gamma<2s-1,\\
					=0 &\text{if } \gamma=2s-1.\\
				\end{cases}
	\]
	Observe that, by \eqref{tt01}, in the previous estimate, we have used that 
	$|\rho_0(x)|=|\rho_1(x)|$ for all $x\in B_1(0)$. Thus the conclusion follows in the case $\mu>0$.

		\medskip

		\noindent {\it Case 2: $\mu=0.$} As now $t_0=-\sqrt{d(x)},$ $t_1=\sqrt{d(x)}$, by symmetry of the integral it is follows that
		$$
		\begin{aligned}
		J&=2 \int_{0}^{\sqrt{d(x)}}\dfrac{(1-|x|^2-\tau^2)_{+}^\gamma
		-(1-|x|^2)^\gamma}{|\tau|^{1+2s}}d\tau\\
    &=2w(x)d(x)^\gamma
\int_{0}^{\sqrt{d(x)}}\frac{\left(1-\frac{\tau^2}{d(x)}\right)^{\gamma}-1}{|\tau|^{1+2s}}\, d\tau\\
&=\frac{w(x)}{d(x)^{s-\gamma}}\int_{0}^{1}\frac{(1-z)^{\gamma}-1}{|z|^{1+s}}\, dz<-C\frac{w(x)}{d(x)^{s-\gamma}},
\end{aligned}
$$
for some $C>0$ as wanted.
	\end{proof}

	Let us define the second barrier
		\[
			v_\gamma(x):=|x|^\gamma,  \quad \gamma\in\mathbb{R},
		\]
for which we have the next results.
	\begin{lemma}\label{barrera_1}
		For every
		$0<\gamma<2s-1$ there exists
		 $C>0$ such that
		\[
			\mathcal{I}_{\hat{x}} (v_\gamma)(x)\le -C|x|^{\gamma-2s},\, 
			x\in \mathbb{R}^{N}\setminus\{0\},
		\]
		where $\hat{x}=\tfrac{x}{|x|}\in \mathbb{S}^{N-1}.$
	\end{lemma}
	\begin{proof}
		Let $x\in \mathbb{R}^{N}\setminus\{0\},$ and $\gamma\in\left(0,2s-1\right)$. 
		By using the change of variable $t=\tfrac{\tau}{|x|},$ we have that
		\[
			\begin{aligned}
				\mathcal{I}_{\hat{x}} (v_\gamma)(x)&=|x|^{\gamma-2s}
				\int_{\mathbb{R}}
				\dfrac{ |1+t|^\gamma-1}
				{|t|^{1+2s}}dt,
			\end{aligned}
		\]
		where the last integral is finite and negative by using, one more time, the ideas of \cite[Proof of Lemma B.1]{BCGJ}.
	\end{proof}

\begin{lemma} \label{cota1+} 
    For every $0<\gamma<2s$ and $y_0\in\mathbb{R}^{N}$ there exists $C > 0$ such that
    $$
        \Lambda_N^s (|x-y_0|)^\gamma\leq C |x-y_0|^{\gamma-2s}\ ,\,x\in\mathbb{R}^{N},
     $$
    here $\Lambda_N^s$ is a maximal operator defined as
    \[
	    \Lambda_N^s (v)(x)\coloneqq\sup_{\theta\in\mathbb{S}^{N-1}}\int_{\mathbb{R}}
    \dfrac{{v}(x+\tau \theta)-{v}(x)}{|\tau|^{1+2s}} d\tau.
    \]
	\end{lemma}
\begin{proof} Follows the same lines of the proof of \cite[Lemma 3.3]{DQT}.
      \end{proof}

	To prove our last auxiliary result we have to take into account the geometry properties of our domain. In particular we will use that, by \cite[Lemma 6.2]{BGI},
	\begin{equation}\label{geometria}
		\Omega= \bigcap\limits_{z\in \partial\Omega}B_R(z-R\nu(z)),
	\end{equation}
for some $R>0$ whose value it is related with the principal curvatures of $\partial\Omega$ and $\nu(z)$ denotes the outward normal unit vector of $\Omega$ in $z\in\partial\Omega$.
%
Now we can establish the following result.
	\begin{lemma} \label{cota11}
		Let $u$ be a {viscosity}  solution  of \eqref{eq:regularidad} with  $f$ bounded 
		and $g$ satisfying \eqref{holder-g}. Then there exists $C>0$ such that
 		\[
 		    u(x)-g(z)\geq -C \mbox{dist}(x,\partial\Omega)^\gamma, \quad x\in\overline{\Omega}
 		\]
        for all $z \in \partial \Omega$ verifying that $\mbox{dist}(x,\partial\Omega)=|x-z|$ and with
        \[
            \gamma\in \begin{cases}
	           (0,2s-1) , &\text{ if }g\equiv0,\\
	            ({0},\beta-s) , &\text{ if }g\neq0.
        \end{cases}
        \]	
	\end{lemma}

	\begin{proof} 
	    Let us first consider the case $g\equiv0$. For that let $\gamma\in(0,2s-1),$ 
	    $z\in \partial\Omega$ and $M>0$ 
	    a positive constant that will be chosen later. We define the non-positive function
		\[
			g_{z,M}^\gamma(x)\coloneqq -M\left(R^2-|x-y_z|^2\right)_{+}^\gamma,
			\quad\mbox{where}\quad y_z:=z-R\nu(z),\quad x\in\mathbb{R}^{N}.
		\] 
		Since, by Lemma \ref{barrera_2}, we know that there exists $C>0$ such that
		\begin{equation}\label{desp0}
			-\Lambda_1^s g_{z,M}^\gamma(x)\le 
			-CM\left(R^2-|x-y_z|^2\right)^{\gamma- {s}},\quad x\in B_R(y_z),
		\end{equation}
		using the fact that $\gamma-{s}<0,$ we get
		\[
			-\Lambda_1^s g_{z,M}^\gamma(x)\le -CMR^{2(\gamma- {s})},\quad
			x\in B_R(y_z).
		\]
		Taking now $M\ge \tfrac{\|f\|_{\infty}}{CR^{2(\gamma- {s})}},$
		 (independent of $z$), it is clear that
		\[
			-\Lambda_1^s g_{z,M}^\gamma(x)\le -\|f\|_{L^{\infty}(\Omega)}
			\le -\Lambda_1^s u(x), \quad 
			x \in {\Omega}.
		\]

		On the other hand, since by \eqref{geometria} $\Omega\subseteq B_R(y_z)$, 
		we also have that $g_{z,M}^\gamma(x){\leq}0=u(x)$ for any $x\in\mathbb{R}^N\setminus {\Omega}\,$. Thus the comparison principle given in \cite{DPQR_FC} allows us to affirm
		\[
			u(x)\ge g_{z,M}^\gamma(x),\quad x\in\mathbb{R}^N.
		\]
		
		Let us now consider a fixed but arbitrary $x\in \Omega$. Taking $z\in\partial\Omega$ such that
		$|z-x|=\mbox{dist} (x, \partial\Omega)$ by using the Cauchy–Schwarz inequality we conclude that
		\begin{equation}\label{despp}
				\begin{aligned}
				g_{z,M}^\gamma(x)&=-M\left(R^2-|x-z+R\nu(z)|^2\right)^\gamma\\
				&=-M\left(-|x-z|^2-2R(x-z)\cdot\nu(z)\right)^\gamma\\
				&=-M\left||x-z|^2+2R(x-z)\cdot\nu(z)\right|^\gamma\\
				&\ge-M\left(|x-z|^{2\gamma}+2R|x-z|^\gamma\right)\\
				&=-M|x-z|^\gamma\left(|x-z|^{\gamma}+2R\right)\\
				&\ge-M\mbox{dist} (x, \partial\Omega)^\gamma\left(\max_{y\in\Omega,\, 
					z\in\partial\Omega}|y-z|^{\gamma}+2R\right),
			\end{aligned}
		\end{equation}
	
		as wanted.
		
		\medskip
		
		In the case $g\not=0,$ 
		let {us consider a fixed} $\gamma\in(0,\beta-s)$ {that will be selected later.  For} every $x\in\mathbb{R}^{N}$ we define the auxiliary function,
		$$
		    h(x)\coloneqq g(z)-M_g|x-z|^{\beta}+\mu g_{z,M}^\gamma(x),\quad z\in\partial\Omega,
		$$
		where $M_g$ is the H\"older constant given in \eqref{holder-g} and $\mu=\mu(R)>0$ is a fixed, 
		but arbitrary, parameter that will be taken later. Applying the standard extremal inequality
		$$
		    -\Lambda_1^sv_1\leq- \Lambda_1^sv_2+ \Lambda_N^s(v_2-v_1),
		$$ 
		for $v_1=h(x)$ and $v_2=\mu g_{z,M}^\gamma(x)$, Lemma \ref{cota1+} and \eqref{desp0}, 
		we get
		$$
		    -\Lambda_1^sh(x)\leq M_g|x-z|^{\beta-2s}-{CM\mu|x-z|^{\gamma-s}},
		$$
		with $x\in B_{R}(y_z),\, y_z=z-R\nu(z)$. Thus by using the fact that $\beta\in(s,2s)$ 
		by choosing  $\gamma<{\beta-s}$ and $\mu$ large enough, we obtain that 
		\[
		    -\Lambda_1^sh(x)\leq -\|f\|_{L^{\infty}(\Omega)} 
		    \text{ in } B_R(y_z). 
		\]
		Since, by \eqref{holder-g},  
		it is also true that $h(x)\leq g(x),\, x\in \mathbb{R}^{N}\setminus B_R(y_z)$ 
		we obtain that the function $h$ is a subsolution of \eqref{eq:regularidad}. Thus, by comparison 
		we have that $h(x)\leq  u(x)$ for all $x \in  \mathbb{R}^N$. 
		Doing as we did in the case $g=0$ we consider now 
		$z$ such that $|z-x|=\mbox{dist} (x,\partial\Omega)$. 
		So that, by using \eqref{despp} we conclude 
		$$
		    u(x)-g(z)\geq  h(x)-g(z)\geq -C\mbox{dist} (x, \partial\Omega)^\gamma,
		$$
		as wanted.		
	\end{proof}
	
	We are now able to conclude this section by proving the regularity result. 
	\begin{demoregularidad}
		We proceed again as in \cite{BGI}.
		Let us to consider some points, fixed but arbitrary, $x_0,y_0\in\bar\Omega$. We have 
		to prove that there exists $L>0$ such that
	    \begin{equation}\label{puntito}
	        u(x_0)\leq u(y_0)+L|x_0-y_0|^{\gamma},
	    \end{equation}	
	    with
	    \[
            \gamma\in \begin{cases}
	           (0,2s-1) , &\text{ if }g\equiv0,\\
	            ({0},\beta-s) , &\text{ if }g\neq0.
        \end{cases}
        \]	 
	     For that, on one hand we notice that, by Lemma \ref{barrera_1}, 
	     there exists $C>0$, independent of $y_0,$ satisfying
        \[
            -\Lambda_1^s H_{y_0}(x) \ge CL|x-y_0|^{\gamma-2s},\, x\neq y_0,
         \] 
         where 
         \[
            H_{y_0}(x):=u(y_0)+L|x-y_0|^{\gamma},\quad x\in\mathbb{R}^{N}.
          \] 
          Therefore we can choose $L>0$ in order to have
            \begin{equation}\label{ppa}
                -\Lambda_1^s H_{y_0}(x) 
                \ge CL\delta^{\gamma-2s}\geq\|f\|_{L^{\infty}(\Omega)}
                \geq -\Lambda_1^s u(x),\quad |x-y_0|<\delta,\, x\in\Omega,
            \end{equation}
for a fixed but arbitrary $0<\delta<1$. Let us now prove that
\begin{equation}\label{ppa2}
u(x)\leq H_{y_0}(x)\quad\mbox{ in $\mathbb{R}^N\setminus (B_{\delta}(y_0)\cap\Omega)$.}
\end{equation}
Indeed if $x\in \Omega\cap (\mathbb{R}^N\setminus B_{\delta}(y_0))$ we can pick $L>0$ large enough such that \eqref{ppa2} is satisfied by using the fact that $u$ is a bounded function and $\gamma>0$. Moreover when  {$g\equiv 0$ and} $x\in\mathbb{R}^N\setminus \Omega$ applying Lemma \ref{cota11} we know that there exists $C>0$, independent of $y_0$, such that
$$u(x)-u(y_0)=-u(y_0)\leq C\mbox{dist}(y_0,\partial\Omega)^{\gamma}\leq C|x-y_0|^{\gamma}\leq L|x-y_0|^{\gamma},$$
by choosing $L\geq C$. Therefore, from \eqref{ppa} and \eqref{ppa2}, applying one more time the comparison principle given in \cite{DPQR_FC} we conclude that
$$u(x)\leq H_{y_0}(x),\quad x\in\mathbb{R}^{N},$$
that implies \eqref{puntito} by taking $x=x_0\in\bar{\Omega}.$

For the general case $g\not =0$ we borrow some ideas of \cite{DQT}. For that we take $y_0\in\overline{\Omega}$, and {a small parameter $\nicefrac12>{\delta}$.} For some $L>0$, fixed but arbitrary, we define the function 
$$\Phi (x):= u(x)-u(y_0)-L|x-y_0|^\gamma.$$
We claim that, for $L$ large enough, just in terms of the data and ${\delta}$, but not on $y_0$, $\Phi$ is non-positive in $\mathbb{R}^{N}$. We notice that, from this, we get that
$$u(x)-u(y_0)\leq L|x-y_0|^\gamma \quad \mbox{for all}\quad {x\in}\mathbb{R}^{N},$$
so, since $L$ does not depend on $y_0$, interchanging the role of $y_0$ and $x$ we conclude the $\gamma$-H\"older regularity of $u$ as wanted.
Suppose by contradiction that the previous claim is false, that is 
\begin{equation}\label{vv}
\sup_{\mathbb{R}^{N}} \Phi> 0.
\end{equation}
With an initial choice 
{$L \geq 2\nicefrac{\|u\|_\infty}{{\delta}^\gamma}$
it is follows that the supremum, that has to be positive, is achieved at some point $\bar x\neq y_0$, $\bar x\in B_{{\delta}}(y_0)$. Let us now prove that this would be impossible by the fact that $\bar x$ cannot belong to $\Omega$ nor to $\mathbb{R}^N\setminus\Omega$. Indeed, if $\bar x\in\Omega$} we consider $\phi(x)\coloneqq H_{y_0}(x) \coloneqq u(y_0){+}L|x-y_0|^\gamma$ as a test function  (regarded as a subsolution to the problem) at the maximum point $\bar x$ 
to obtain
$$-\Lambda_1^s H_{y_0}(\bar x) \le f,$$
{that contradicts \eqref{ppa} by taking $L>0$ large enough.
    Let us now conclude by proving that $\bar x \in  \mathbb{R}^{N}\setminus \Omega$ is not     
    possible too}. In fact, {if $\mbox{dist} (y_0,\partial\Omega):=|y_0-z|,\, z\in\partial\Omega$}, 
    since $\gamma<\beta-s$, by Lemma \ref{cota11}, we have
    \[
        u(\bar x)-u(y_0)\leq u(\bar x)-g(z)+g(z)-u(y_0) \leq g(\bar x)-g(z)+C \mbox{dist}({y_0},
        \partial\Omega)^\gamma,\, \bar x \in  (\mathbb{R}^{N}\setminus \Omega),
     \]
    and by the H\"older condition for $g$ we get
    \begin{equation}\label{desppp}
        u(\bar x)-u(y_0)\leq M_g |  \bar x-z |^\beta+C |y_0-z|^\gamma 
        \leq  \tilde C |\bar x -y_0|^\gamma.
    \end{equation}
    We notice that to obtain the last inequality we have used  that $\gamma<\beta$ and  
    the simple facts that  $|y_0-z| \leq |y_0-\bar x|<\delta<1$ and, therefore, 
    $|\bar x -z|\leq 2 |\bar x -y_0|(<2\delta<1)$.	
We conclude by noticing that \eqref{desppp} is also not possible by taking 
$L \geq \tilde C$ since this implies 
$\Phi(\bar x)\leq 0$ against the hypothesis \eqref{vv}. 
So that the regularity follows.
	\end{demoregularidad}

\begin{remark} {\rm
A similar result concerning with the regularity up to the boundary for truncated fractional laplacians operators with zero boundary datum can be also find in \cite[Proposition 7.1]{biri}. Note that both results, ours and the one obtained in the aforementioned work, do not show that this regularity is optimal for these operators.}
\end{remark}


\section{Eigenvalue problem} \label{sect-eigen}

	Before starting the study of the eigenvalue problem, we need to establish a maximum principle given in the next	lemma 
	\begin{lemma}[Strong Maximum Principle] \label{SMP}
		Let $u$ be a non-positive  viscosity solution of 
		$$
		 \begin{cases}
		-\Lambda_1^s u(x)\leq 0 & \mbox{in}\quad \Omega,\\
		u(x)=0 & \mbox{in}\quad \mathbb{R}^N\setminus\Omega.
		\end{cases}		
$$ Then $u =0$ or $u<0$.
	\end{lemma} 
	\begin{proof} 
		Let us suppose by contradiction that there exits $x_0 \in \Omega$ 
		such that $u(x_0)=0$ and $u<0$ in $B_r(x_1)\subset \Omega$ 
		for some $r>0$, $x_1\in\Omega$.
 		We take now $\varphi\equiv0$ as the regular function in the notion of viscosity solution ($u \leq \varphi\equiv 0$).
		Thus, if we define 
		$$
		    \theta\coloneqq\frac{x_1-x_0}{|x_1-x_0|}\in\mathbb{S}^{N-1}
		 $$ 
		and we chose $\delta>0$ small enough such that
 		$B_\delta(x_0)\subset \Omega\setminus B_r(x_1)$, we get 
 		\[
			\int_{\mathbb{R}\setminus(-\delta,\delta)}
			\frac{-{u }(x_0+\tau \theta)}{|\tau|^{1+2s}} d\tau>0.		
		\]
  		Since $\varphi\equiv 0$ the previous inequality clearly implies
  		\begin{equation}\label{simpli}
  			\sup_{z\in\mathbb{S}^{N-1}}\int_{-\delta}^{\delta}\frac{\varphi(x_0)-{\varphi} (x_0+\tau z)}
			{|\tau|^{1+2s}} d\tau+ \int_{\mathbb{R}\setminus(-\delta,\delta)}
			\frac{\varphi(x_0)-{u }(x_0+\tau z)}{|\tau|^{1+2s}} d\tau
			>0.
		\end{equation}
		By using the fact that $x_0$ is a local maximum for 
		$u-\varphi$ and $-\Lambda_1^s u\leq 0$ in the viscosity sense, we also get 
		$$
			-\inf_{z\in\mathbb{S}^{N-1}}\int_{-\delta}^{\delta}\frac{-\varphi(x_0)+{\varphi} (x_0+\tau z)}
			{|\tau|^{1+2s}} d\tau+ \int_{\mathbb{R}\setminus(-\delta,\delta)}
			\frac{-\varphi(x_0)+{u }(x_0+\tau z)}{|\tau|^{1+2s}} d\tau
			\leq 0,
		$$
		that is in contradiction with \eqref{simpli}.
	\end{proof}
	
	We can now show the main result of this section.	
	\begin{theorem} \label{autof}
	    Assume $s>\nicefrac12.$
		There exist $\mu_1>0$ and $\varphi_1\in\mathcal{C}(\mathbb{R}^{N})$ negative in 
		$\Omega$
		solving the Dirichlet problem
		$$
			\begin{cases}
				-\Lambda_1^s \varphi_1(x)=\mu_{1}\varphi_1 
				&\text{in }\Omega,\quad\\
			 \varphi_1=0&\text{in } \mathbb{R}^{N}\setminus\Omega,\\
				\end{cases}
		$$
		in the viscosity sense.
	\end{theorem}

	\begin{proof}
		The idea is to apply the non-linear version of the Krein-Rutman Theorem 
		(see \cite{Ara, Maha}) in the real Banach space of continuous functions  
		$X\coloneqq \mathcal{C}(\overline\Omega)$, with  the uniform norm. 
		The cone will be the set of non-positive continuous functions 
		$K\coloneqq\mathcal{C}_{-}(\overline\Omega)$. 
		Notice that this produce a ``reverse inequality" since the standard is the nonnegative cone. 
		In fact, given $x,y \in X$  $ x\preceq y \iff y-x \in \mathcal{C}_{-}(\overline\Omega)$ 
		so  $y\leq x$ in $\overline\Omega$. For the  definition of 
		$\mathcal T$ we proceed to invert the operator $-\Lambda_1^s$. 
		For every $v\in\mathcal{C}(\overline{\Omega})$ we know that 
		there exists an unique viscosity solution $u\in\mathcal{C}(\overline{\Omega})$ of the problem
		$$
			\begin{cases}
			-\Lambda_1^s u=v &\text{in }\Omega,\quad\\
			u=0&\text{in } \mathbb{R}^{N}\setminus\Omega,\\
				\end{cases}
		$$
		{(see \cite[Theorem 1.1]{DPQR_FC})} then the operator 
		$\mathcal T\colon\mathcal{C}(\overline\Omega)\to 
		\mathcal{C}(\overline\Omega)$ given by 
		$\mathcal{T}(v)\coloneqq u$, is well defined.  
		Moreover, the operator $-\Lambda_1^s$ satisfies comparison principle 
		{(see \cite[Theorem 2.2]{DPQR_FC})}  
		thus $\mathcal{T}$ is increasing and also positively $1-$homogeneous.  
		With  the same arguments as in  \cite[Lemma 4.5]{CS}  (see also \cite[Corollary 4.7]{CS})  {we obtain that} $\mathcal{T}$ is continuos with the uniform norm.

		Moreover using the trivial function as a super solution and comparison we get that if 
			$v\in \mathcal{C}_{-}(\overline\Omega)$ then  $T( {v})\in \mathcal{C}_{-}(\overline\Omega)$.
	So 
	$$
		\mathcal T:\mathcal{C}_{-}(\overline\Omega)
		\to \mathcal{C}_{-}(\overline\Omega),
	$$
	is well defined.
	
	 Now by Theorem \ref{regularidad.intro} we also know that if $v_n$ is a  {uniformly bounded} sequence then
	$$
		\|\mathcal T(v_n)\|_{\mathcal{C}^{\beta}
		(\overline\Omega)}\leq C,
	$$
	for some $C>0$ (independent of  $n$) and  
	for some $\beta\in (0,1)$, thus we get that the operator is 
	compact using Ascoli-Arzela Theorem.

	If we take $v_0 \in K$ non trivial with compact support in $\Omega$ then $\mathcal T(v_0)<0$ 
	by the strong maximum principle so there exists  $M$ large such 
	that  $M\mathcal  T(v_0)\leq v_0$. 
	Therefore we have checked all the requirements to apply the Krein-Rutman 
	Theorem that allows us to conclude the existence of 
	$\varphi_1\in\mathcal{C}_{-}(\overline\Omega)$, 
	$\varphi_1\neq 0$ and $\mu_1>0$ such 
	that $\mathcal T(\varphi_1)=\mu_1\varphi_1$. 
	This implies the existence of an eigenvalue for 
	$-\Lambda_1^s$ associated to a  {nonpositive} eigenfunction 
	$\varphi_1$ as we wanted to show. Finally, notice that $\varphi_1<0$ by the strong maximum principle given in Lemma \ref{SMP}.
	\end{proof}
	
	\begin{remark}
		{\rm It is important to bear in mind that, as occurs in the local case, the simplicity of the eigenvalue found in the previous results is nowadays an open problem. }
	\end{remark}
\section{Evolution problem} \label{sect-evolution}
	This section is devoted to study the evolution problem given in 
	\eqref{ev_eq} under some conditions on the datum. We want to emphasize that 
	we restrict  our attention to the case when the outside datum does not depends on the time variable 
	$t$. Notwithstanding the above, we observe that the comparison principle presented in the next 
	theorem will be also true in the case when $g(x,t)$ as well as the existence result whenever we 
	require a Lipschitz assumption, with respect to $t$, on the function. 
	We start by introducing the notion of viscosity solution we will work with.
	
	\medskip
	
	Let $T>0.$ A bounded upper (resp. lower) semicontinuous function
	$u\colon\mathbb{R}^N\times[0,T]\to\mathbb{R}^N$
	is said to be a viscosity subsolution  (resp. supersolution) of 
	\begin{equation}\label{ev_eq_T}
		\begin{cases}
			u_t (x,t) =\Lambda_1^s u(x,t) &\text{in }\Omega\times(0,T),\\
			u(x,t)=g(x) &\text{in }
				(\mathbb{R}^N\setminus\Omega)\times\textcolor{black}{[}0,T),\\
			u(x,0)=u_0(x) &\text{in }\textcolor{black}{\Omega},
		\end{cases}
	\end{equation}
	if
	\begin{itemize}
		\item $u(x,t)\leq g(x)$ \mbox{ {(resp. $\geq$)}} in $(\mathbb{R}^N\setminus\Omega)\times\textcolor{black}{[}0,T);$
		\item $u(x,0)\leq u_0(x)$ \mbox{ {(resp. $\geq$)}} in $\textcolor{black}{\Omega};$
		\item For every smooth function 
			$\varphi\colon\mathbb{R}^N\times[0,\infty)\to \mathbb{R}$
			in every maximum (resp. minimum) point $(x_0,t_0)\in\Omega\times (0,T)$ of
			$u-\varphi$ in $B_\delta(x_0)\times (t_0-\delta,t_0+\delta)
			\subseteq \Omega\times (0,T)$ for some $\delta>0$ we have that
			\[
				 {u^*_t}(x_0,t_0)-\Lambda_1^s {u^*}(x_0,t_0)\le0\, (resp. \ge0)
			\]
			where
			\[
				 {u^*}(x,t)=
				\begin{cases}
					\varphi(x,t) &\text{ if } (x,t)\in
						B_\delta(x_0)\times (t_0-\delta,t_0+\delta),\\
					u(x,t) &\text{ if } (x,t)\not\in
						B_\delta(x_0)\times (t_0-\delta,t_0+\delta).
				\end{cases}
			\]
	\end{itemize}
	
	A bounded function $u\colon \mathbb{R}^N\times[0,T]\to\mathbb{R}^N$ is a
	viscosity solution of \eqref{ev_eq_T} if it is both subsolution and supersolution.
	
	\medskip
	
	Then, a bounded upper (resp. lower) semicontinuous function
	$u\colon\mathbb{R}^N\times[0,\infty) \to\mathbb{R}^N$
	is said to be a viscosity subsolution  (resp. supersolution) of \eqref{ev_eq} if
	$u$ is  a viscosity subsolution  (resp. supersolution) of \eqref{ev_eq_T} for any 
	$T>0.$ Consequently, a continuous function 
	$u\colon \mathbb{R}^N\times[0,\infty)\to\mathbb{R}^N$ is a
	viscosity solution of \eqref{ev_eq} if it is both subsolution and 
	supersolution.
	
	\medskip
	
	 	
	Our next aim is to show that, under reasonable hypothesis, there exists a unique viscosity solution of
	\eqref{ev_eq}, that is, to obtain the proof of Theorem \ref{teo-1.intro}. The proof is obtained as usual: first, we show a comparison 
	principle and then we conclude by the Perron's method using adequate sub and supersolutions.

	\begin{theorem} \label{comparacion}
		Let $T>0$, $u$ and $v$ be  viscosity subsolution and supersolution of 
		\eqref{ev_eq_T} respectively where $\Omega$ is a convex domain, then
		$u\le v$ on $\Omega\times [0,T]$.
	\end{theorem}
	
	\begin{proof}
 		Following the  ideas given in \cite{CIL}, 
 		we 
 		define the function
 		$$
 			\widetilde{u}(x,t)\coloneqq u(x,t)-
 			\tfrac{\varepsilon}{T-t},\quad\varepsilon>0,
 		$$
  		that satisfies, in the viscosity sense,
  		$$
  			\begin{cases}
			\widetilde {u}_t(x,t)-\Lambda_1^s \widetilde{u}(x,t)\leq-
			\tfrac{\varepsilon}{(T-t)^2}<0 &\text{in }\Omega\times(0,T),\\
			\widetilde{u}\to -\infty, &\text{when }
				t\to T^{-}.\\
				\end{cases}
		$$
		It is clear that, if we prove 
		$\widetilde{u}(x,t)\leq v(x,t)$ for $x\in\Omega$ and $t\in (0,T)$ 
		then the conclusion follows by taking $\varepsilon\to 0$.  
		Therefore let us assume that 
  		$$
  		 	\begin{cases}
				u_t(x,t)-\Lambda_1^s u(x,t)<0 &\text{in }\Omega\times(0,T),\\
				u\to -\infty, &\text{when }
					t\to T^{-},\\
			\end{cases}
		$$
 		and we suppose, by contradiction, that
 		$$
 			\sup_{x\in\overline{\Omega},\, t\in [0,T)}{u(x,t)-v(x,t)}
 			\coloneqq\eta>0.
 		$$
 
 		Since $u$ and $-v$ are bounded from above, 
 		$\overline\Omega$ is compact and $u\to -\infty$ as $t\to T^{-}$ then,  
 		for some $\alpha$ big enough that will be chosen later, there exists 
 		$(x_{\alpha}, y_{\alpha}, t_{\alpha}, r_{\alpha})$ the maximum point of
 		$$
 			G_{\alpha}(x,y,t,r)\coloneqq u(x,t)-v(y,r)-
 			\alpha(|x-y|^{2}+|t-r|^2),
 		$$
		with $(t,r)\in [0,T)^2,
 			 (x,y)\in \overline{\Omega}\times \overline{\Omega}.$
 		If we denote
 		$$
 			M_{\alpha}\coloneqq 
 			G_{\alpha}(x_{\alpha}, y_{\alpha}, t_{\alpha}, r_{\alpha})(\geq \eta>0),
 		$$
 		it is clear that 
 		$M_{\alpha}\searrow \widetilde{M}$ as $\alpha\to\infty,$ 
 		for some $\widetilde{M}\geq\eta>0$. 
 		Moreover, since
 		$$
 			M_{\frac{\alpha}{2}}
 			\geq G_{\frac{\alpha}{2}}(x_{\alpha}, y_{\alpha}, t_{\alpha}, 
 			r_{\alpha})
 			=M_{\alpha}+\frac{\alpha}{2}(|x_{\alpha}-y_{\alpha}|^{2}
 			+|t_{\alpha}-r_{\alpha}|^2),
 		$$
		it follows that 
		$$
			\lim_{\alpha\to\infty}\alpha |x_{\alpha}-y_{\alpha}|^{2}
			=\lim_{\alpha\to\infty}
			\alpha |t_{\alpha}-r_{\alpha}|^{2}= 0.
		$$
		We notice that the previos limits also implies that, if, by subsequence,
		\[
			\lim_{\alpha\to\infty}(x_\alpha,y_\alpha)\coloneqq
			(\widetilde x, \widetilde y)
			\in \overline{\Omega}\times \overline{\Omega},\text{ and } 
			\lim_{\alpha\to\infty}(t_\alpha,r_\alpha)=(\widetilde t, \widetilde r)
			\in [0,T)\times [0,T),
		\]
		then $\widetilde x=\widetilde y$ and $\widetilde t=\widetilde r.$
		We also observe that, since
		$$
			0<\eta\leq\widetilde{M}=\lim_{\alpha\to\infty}
			G_{\alpha}(x_{\alpha}, 
			y_{\alpha}, t_{\alpha}, r_{\alpha})
			\leq u(\widetilde x, \widetilde t)-v(\widetilde x, 
			\widetilde t)\leq \eta,
		$$
		clearly 
		\begin{equation}\label{mg}
		\eta=u(\widetilde x, \widetilde t)-v(\widetilde x, \widetilde t).
		\end{equation}
		By using the facts that $u(\cdot,t)\leq v(\cdot, t)$ in $\partial\Omega$ 
		and $u(x,0)=v(x,0),$ \eqref{mg} implies that 
		$\widetilde x=\widetilde y\notin \partial\Omega$ and 
		$\widetilde t=\widetilde r\neq 0$. Therefore, we can define
		\begin{equation}\label{interior}
			d_{\alpha}
			\coloneqq\min\{d(x_\alpha,\partial\Omega),\, 
			d(y_\alpha,\partial\Omega)\}>0,\mbox{ and }\delta_{\alpha}
			\coloneqq\min  \{t_{\alpha},\, r_{\alpha}\}>0.
		\end{equation}
		
		Without loss of generality, we can also assume that
		\begin{equation}\label{tecnico}
			\frac{d(\widetilde{x},\partial\Omega)}2<d_\alpha\text{ and }
				0<t_{\alpha}\leq r_{\alpha}.
		\end{equation} 
		
		We define now the regular function
		$$
			\psi_{\alpha}(z,t)\coloneqq v(y_{\alpha}, r_{\alpha})
			+\alpha(|t-r_{\alpha}|^2+|z-y_{\alpha}|^2),
		$$ 
		that clearly satisfies 
		\begin{equation}\label{entiempo}
			\partial_{t}\psi_{\alpha}(x_{\alpha}, t_{\alpha})=2\alpha(t_\alpha-r_\alpha).
		\end{equation} 
		Moreover $\psi_{\alpha}(z,t)$ touches $u$ from above in $(x_\alpha,t_{\alpha})$. 
		Indeed, by \eqref{interior}, if we consider $w\in\mathbb{R}^{N}$ and $\delta>0$
		such that $x_{\alpha}+w\in \Omega$ and $t_{\alpha}+\delta\in [0,T)$ since 			\[
			G_{\alpha}(x_{\alpha}, y_{\alpha}, t_{\alpha}, r_{\alpha})\geq 
			G_{\alpha}(x_{\alpha}+w, y_{\alpha}, t_{\alpha}+\delta, r_{\alpha}),
		\]
		 by \eqref{tecnico} we get that
		$$
			u(x_\alpha,t_{\alpha})-\alpha|x_{\alpha}-y_{\alpha}|^2
			-u(x_{\alpha}+w,t_{\alpha}+\delta)+\alpha|x_{\alpha}+w-y_{\alpha}|
			\geq -\alpha\delta^2.
		$$
		Thus
		$$
			u(x_{\alpha},t_{\alpha})-\psi_{\alpha}(x_{\alpha},t_{\alpha})
			\geq u(x_{\alpha}+w,t_{\alpha}+\delta)-\psi_{\alpha}(x_{\alpha}
			+w,t_{\alpha}+\delta),
		$$
		as wanted. Therefore, since $u$ is a subsolution, 
		by \eqref{entiempo} it follows that
		\begin{equation}\label{uno}
				0\geq 
				2\alpha(t_\alpha-r_\alpha)-\Lambda_1\psi_{\alpha}(x_{\alpha},t_{\alpha})
				=-\inf\{E_{z,\delta}(u,\psi_\alpha,x_\alpha,t_\alpha)\colon 
				z\in\mathbb{S}^{N-1}\}.
		\end{equation}
		for any $\delta<d_\alpha.$ Here
		$$
		\begin{array}{l}
			\displaystyle E_{z,\delta}(w,\phi,x,t)\coloneqq
			I^1_{z,\delta}(\phi,x,t)+
			I^2_{z,\delta}(w,x,t)+I^3_{z}(w,x,t)-2\alpha(t_\alpha-r_\alpha),\\[7pt]
			\displaystyle I^1_{z,\delta}(\phi,x,t)\coloneqq\int_{-\delta}^\delta
			\frac{\phi(x+\tau z,t)-\phi(x,t) }{|\tau|^{1+2s}} d\tau,\\[7pt]
			\displaystyle I^2_{z,\delta}(w,x,t)\coloneqq\int_{A_\delta^{z}(x)}
			\hspace{-.6cm}\frac{w(x+\tau z,t)-w(x,t) }{|\tau|^{1+2s}} d\tau
			\\[7pt] \qquad \text{ where } 
			A_\delta^{z}(x)\coloneqq \{t\in \mathbb{R}
			\setminus(-\delta,\delta)\colon x+tz\in\Omega\},\\[7pt]
			\displaystyle I^3_{z}(w,x,t)\coloneqq\int_{B^{z}(x)}
			\hspace{-.6cm}
			\frac{g(x+\tau z,t)-w(x,t) }{|\tau|^{1+2s}} d\tau \\[7pt]
			\qquad \text{ where } 
			B^{z}(x)\coloneqq \{t\in \mathbb{R}\colon x+tz\not\in\Omega\},
		\end{array}
		$$
		for $w$ and $\phi$ regular functions and $(x,t)\in\mathbb{R}^N\times(0,T).$
		Defining now 
		$$
			\varphi_{\alpha}(z,t)\coloneqq
			u(x_{\alpha}, t_{\alpha})-\alpha(|t-t_{\alpha}|^2+|z-x_{\alpha}|^2),
		$$
		reasoning in a similar way, we also get that $\varphi_{\alpha}(z,t)$ touches 
		$v$ from below in $(y_\alpha,r_{\alpha})$.
		Therefore, since it is also true that 
		$\partial_t\varphi_{\alpha}(y_{\alpha},r_{\alpha})
		=-2\alpha(r_\alpha-t_\alpha)$, 
		using now the fact that $v$ is supersolution, we obtain
		\begin{equation}\label{dos}
			0 \leq 2(t_\alpha-r_\alpha)-\Lambda_1\varphi_{\alpha}(y_{\alpha},r_{\alpha})
			=-\inf\{E_{z,\delta}(v,\varphi_\alpha,y_\alpha,r_\alpha)\colon 
				z\in\mathbb{S}^{N-1}\}.
		\end{equation}
		for any $\delta<d_{\alpha}.$

		Therefore, by \eqref{uno} and \eqref{dos}, for each $h>0$ 
		there exists $z_{\alpha,h}\in \mathbb{S}^{N-1} $ 
		such that
		\[
			E_{(z_{\alpha,h}),\delta}(u,\psi_\alpha,x_\alpha,t_\alpha)\geq 0,
			\quad \mbox{and}
			\quad  
			E_{(z_{\alpha,h}),\delta}
			(v,\varphi_\alpha,y_\alpha,r_\alpha)
			\leq h,
		\]
		for any $\delta\in(0,d_\alpha),$ so
		\begin{equation}\label{eq:estiinf}
			-h \le E_{(z_{\alpha,h}),\delta}(u,\psi_\alpha,x_\alpha,t_\alpha)- 
			E_{(z_{\alpha,h}),\delta}
			(v,\varphi_\alpha,y_\alpha,r_\alpha)
		\end{equation}
		for $0<\delta<\tfrac{d(\widetilde{x},\partial \Omega )}2.$ 
		We will get a contradiction from the previous estimate by finding an 
		upper bound and passing to the limit in a delicate way. 
		For that, let us first assume 
		\[
		    z_{\alpha,h} \to z_0,\quad\text{when } \alpha\to \infty, \text{ and } h \to 0,
		\]
		by a subsequence 
		if necessary. Let us now observe that there is a positive constant $C$ independent of $\delta,$ 
		$\alpha$ and $h$, such that
        \begin{equation}\label{eq:nueva1}
                  	\max 
			\left\{
				|I^1_{(z_{\alpha,h}),\delta}(\psi_{\alpha},x_\alpha,t_\alpha)|, 
				|I^1_{(z_{\alpha,h}),\delta}(\varphi_{\alpha},y_\alpha,r_\alpha)|, 
			\right\}
			\le C\alpha\delta^{2-2s}.
         \end{equation}		
		Let us now denote by $L_z(x)$ the line that passes trough $x$ with direction $z$, that is
		\[
			L_z (x) \coloneqq x+\tau z,\qquad \tau \in \mathbb{R}.
		\]
		To continue our argument it will be crucial to bear in mind that
		$z_{\alpha,h} \to z_0$ when $\alpha\to \infty,$ and  $h \to 0,$ and the fact that 
		$\Omega$ is a convex domain. Indeed this allows us to affirm that
		\[
			\begin{cases}
				&\mathbf{1}_{L_{z_{\alpha,h}}(y_\alpha)\cap \Omega}, 
				\mathbf{1}_{L_{z_{\alpha,h}}
					(x_\alpha)\cap \Omega} \to 
				\mathbf{1}_{L_{z_0}(\widetilde{x}) 
					\cap \Omega}\quad \mbox{ a.e.,}	\\
				&\mathbf{1}_{L_{z_{\alpha,h}}(x_\alpha) \cap (\mathbb{R}^N \setminus \Omega)}, 
				\mathbf{1}_{L_{z_{\alpha,h}}(y_\alpha)\cap (\mathbb{R}^N \setminus \Omega)} 
				\to \mathbf{1}_{L_{z_0}(\widetilde{x})\cap (\mathbb{R}^N \setminus \Omega)} \quad\mbox{ a.e.},
			\end{cases}
		\]
		when $\alpha\to \infty,$  and $h \to 0.$ 
		Thus since $(x_{\alpha}, y_{\alpha}, t_{\alpha}, r_{\alpha})$ is the maximum point of
 		$G_{\alpha}(x,y,t,r)$ and $u,v$ are bounded continuous functions, by using the previous two limits and the dominated convergence theorem, it follows that
            \begin{equation}\label{eq:nueva2}
                \limsup_{\alpha\to\infty\atop\delta,h\to0}
			I^2_{(z_{\alpha,h}),\delta}(u,x_\alpha,t_\alpha)
			-I^2_{(z_{\alpha,h}),\delta}(v,y_\alpha,r_\alpha)\le 0,
            \end{equation}		
		and
		\begin{equation}\label{eq:nueva3}
                  \lim_{\alpha\to\infty\atop h\to0}	
			I^3_{z_{\alpha,h}}(u,x_\alpha,t_\alpha)
			-I^3_{z_{\alpha,h}}(v,y_\alpha,r_\alpha)
			=-\widetilde{M} \int_{L_{z_0}(\bar x)\cap (\mathbb{R}^N \setminus \Omega)} 
			\frac{d\tau}{|\tau|^{1+2s}}.  
                \end{equation}

		Therefore, 
		from \eqref{eq:nueva1}, \eqref{eq:nueva2}, and \eqref{eq:nueva3}, letting first 
		$\delta \to 0$, then $\alpha \to \infty $, 
		and $h \to 0,$ we get
		\[
			\begin{aligned}
				\limsup\limits_{\alpha\to\infty,\atop \delta,h\to 0}
				E_{(z_{\alpha,h}),\delta}(u, \phi_\alpha, x_\alpha,t_\alpha)&- 
				E_{(z_{\alpha,h}),\delta} 
				(v, \varphi_{\alpha}, y_\alpha,r_\alpha)	\\		
				&\le \displaystyle
						-\widetilde{M} \int_{L_{z_0}(\bar x)\cap (\mathbb{R}^N \setminus \Omega)}
						\frac{d\tau}{|\tau|^{1+2s}}.
			\end{aligned}
		\]
		Thus by \eqref{eq:estiinf}, we conclude
		\[
			0\le -\widetilde{M} 
			\int_{L_{z_0}(\bar x)
			\cap (\mathbb{R}^N \setminus \Omega)} \frac{dt}{|t|^{1+2s}}<0,
		\]
		which implies the desirable contradiction.
	\end{proof}
	
	With the previous result the classical Perron's method can be used so we obtain the next result. Since we assumed that \eqref{hip} holds, by the results obtained in \cite{DPQR_FC}, we know that there exists a unique function $
	 \overline{w}\in\mathcal{C}(\overline\Omega)$ that satisfies, 
	\[
		\begin{cases}
			\Lambda_1^s \overline{w}(x)=-\|\Lambda_1^su_0\|_{L^{\infty}(\Omega)} &\text{in }\Omega,\\
			\overline{w}(x)=g(x) &\text{in }
				\mathbb{R}^N\setminus\Omega,
		\end{cases}
	\]
	in the viscosity sense, with $\overline{w}|_{\partial\Omega}=g|_{\partial\Omega}.$ 
	Since by the comparison principle proved in \cite[Theorem 2.2]{DPQR_FC} it is clear that 
	$\overline{w}\geq u_0$ in $\mathbb{R}^{N}$ we get that $\overline{u}(x,t)\coloneqq\overline{w}(x)$ 
	is a viscosity supersolution of \eqref{ev_eq} where the datum is attained in 
	a classical continuous way. Reasoning similarly we also get that 
	$\underline u(x,t)\coloneqq\underline{w}(x)$ with
	$$
		\begin{cases}
			\Lambda_1^s \underline{w}(x)=\|\Lambda_1^su_0\|_{L^{\infty}(\Omega)} &\text{in }\Omega,\\
			\underline{w}(x)=g(x) &\text{in }
				\mathbb{R}^N\setminus\Omega,\\
		\end{cases}
	$$
	is a subsolution of \eqref{ev_eq} where the datum is also attained in a classical continuous way. That is 
	by assuming \eqref{hip} we obtain a well orderer viscosity sub and supersolution continuous up to $
	\overline{\Omega}$, of problem \eqref{ev_eq}.
	Then, Perron's method gives us the existence of a viscosity solution to \eqref{ev_eq}. Since the uniqueness follows from the comparison principle, Theorem \ref{teo-1.intro} is proved.


	\bigskip 
	
	Once we have proven the existence and uniqueness of viscosity solutions of the parabolic problem, by using 
	the existence of the first eigenvalue and eigenfunction of $-\Lambda_1^s$ (see Theorem \ref{autof}) and 
	following some ideas from \cite{BlancRossi}, we can establish the asymptotic behavior of the solutions 
	of 	\eqref{ev_eq}. For that let us consider 
	\begin{equation}\label{auto1}
		\begin{cases}
			-\Lambda_1^s \varphi_R(x)=\mu_{R}\varphi_R &\text{in }\Omega_{R},\quad\\
			 \varphi_R=0&\text{in } \mathbb{R}^{N}\setminus\Omega_{R},\\
				\end{cases}
	\end{equation}
	and
	\begin{equation}\label{autoN}
		\begin{cases}
			-\Lambda_N^s \psi_R(x)=\mu_{R}\psi_R &\text{in }\Omega_{R},\\
			 \psi_R=0&\text{in } \mathbb{R}^{N}\setminus\Omega_{R},\\
				\end{cases}
		\end{equation}
	where $\Omega\subset\subset\Omega_R$ with $\Omega_R$ is a strictly convex domain of 
	$\mathbb{R}^{N}$, $\mu_R>0$, $0 {<} \psi_R(x)=-\varphi_R(x)$, $x\in\Omega_R$.

	Then we have the following result.
	\begin{theorem}\label{asymp}
		Assume $s>\nicefrac12.$ If $u_0$ and $g$ satisfy \eqref{hip} 
		then there exist $C_i=C_i(u_0)>0$, $i=1,\, 2$ such that 
		\begin{equation}\label{c_uniforme_2}
			C_1e^{-\mu_Rt}\varphi_R(x)\leq u(x,t)-z(x)\leq C_2 e^{-\mu_Rt}\psi_R(x),
			\quad x\in\overline{\Omega},\quad t>0,
		\end{equation}
		where $u(x,t)$ is the unique viscosity solution of \eqref{ev_eq} 
		given by Theorem \ref{teo-1.intro} and $z$ is  {the $s$-convex envelope of $g$, that is,} the unique viscosity solution of 
		\begin{equation}\label{eliptico0}
			\begin{cases}
				-\Lambda_1^s z(x)=0&\text{in }\Omega,\\
				z=g&\text{in } \mathbb{R}^{N}\setminus\Omega.\\
			\end{cases}
		\end{equation}
	\end{theorem}

	\begin{proof}
		For every $x\in\overline{\Omega}$ and $t>0$ let us define, for some positive constants 
		that will be chosen later, the functions
		$$
			\underline u(x,t)\coloneqq z(x)+C_1e^{-\mu_Rt}\varphi_R(x)
			$$
			 and 
			 $$ 
			\overline u(x,t)\coloneqq z(x)+C_2e^{-\mu_Rt}\psi_R(x).
		$$
		On the first hand it is clear that 
		$\underline u(x,t) {\leq}g(x) {\leq}\overline u(x,t)$ in 
		$(\mathbb{R}^{N}\setminus\Omega)\times[0,+\infty)$.  
		By choosing the constants in an appropriate way, depending on the datum $u_0$, since
		$$
			\mbox{$0< \psi_R(x) =-\varphi_R(x)$, (so that $0>\varphi_R(x)$), $x\in{\Omega_R}$,}
			\text{ and }\Omega\subset\subset\Omega_R,
		$$
		it is also true that $\underline u(x,0)\leq u_0(x)\leq \overline u(x,0)$,  {$x\in\mathbb{R}^{N}$}. 
		On the other hand, since for every functions $f,\, g$ we have that,
		$$
			\inf_{x\in\mathbb{R}^N} f(x)+\inf_{x\in\mathbb{R}^N} g(x)\leq \inf_{x\in\mathbb{R}^N} (f(x)+g(x))\leq \inf_{x\in\mathbb{R}^N} f(x)
			+\sup_{x\in\mathbb{R}^N} g(x),
		$$
		then
		$$
			\Lambda_1^{s}\underline u(x,t)\geq \Lambda_1^{s}z(x)+C_1e^{-\mu_Rt}\Lambda_1^{s}\varphi_R(x)
		$$
		and
		$$
			\Lambda_1^{s}\overline u(x,t)\leq \Lambda_1^{s}z(x)+C_2e^{-\mu_Rt}\Lambda_N^{s}\psi_R(x),
		$$
		for every $x\in {\Omega}$ and $t>0$. Therefore, by \eqref{auto1},\eqref{autoN} and \eqref{eliptico0} 
		we obtain that $\underline{u}(x,t)$ (resp. $\overline{u}(x,t)$) is a viscosity subsolution 
		(resp. supersolution) of \eqref{ev_eq}. 
		Thus, the conclusion follows by the comparison principle proved in Theorem \ref{comparacion}.
	\end{proof}
	
		\begin{remark} {\rm
		If we would like to obtain the same bound as the one given in \eqref{c_uniforme_2} 
		replacing $\varphi_R$ (resp. $\psi_R$) by $\varphi_1$ (resp. $\psi_1$), 
		the first eigenfunction in $\Omega$, not in $\Omega_R$, of $\Lambda_1^s$ (resp. $\Lambda_N^s$),  in the general case, 
		we would need to assume that
		$C\varphi_1(x)\leq u_0(x)-z(x)\leq C\psi_1(x),\, x\in\overline{\Omega}$.
		 We notice that, taking into account \cite[Corollary 6.10]{biri} the bounds $\varphi_1(x)\leq -C\delta(x)^{s}$ (resp. $\psi_1(x)\geq C\delta(x)^{s}$) hold true.	
		Therefore, to obtain that 
		$C\varphi_1(x)\leq u_0(x)-z(x)\leq C\psi_1(x),\, x\in\overline{\Omega}$
		we need to assume $- c \delta(x)^s \leq u_0(x)-z(x) \leq c \delta (x)^s$
		that is an extra condition on the data.

		\medskip
		
		Whether we have \eqref{c_uniforme_2} with $(\varphi_R, \psi_R)$ or with 
		$(\varphi_1, \psi_1)$ the conclusion of Theorem \ref{asymp} 
		implies the uniform convergence of the solution of the parabolic problem \eqref{ev_eq} 
		to the $s$-convex envelope of the datum $g$,  {that is, the function} $z$. That is, \eqref{c_uniforme} that we presented in Theorem \ref{teo.conver.intro}.}
	\end{remark}

	In some particular cases we can improve the exponential bound obtained in the previous result. 
	Indeed we have the following result.
	{\color{black}
	\begin{proposition}\label{particular}
		 {Let be $s>1/2$}. If $g$ is an affine function in the whole $\mathbb{R}^N$ then 
		for every $\mu$ there exists $K=K(\mu)>0$
		such that 
		\begin{equation}\label{lonuestro}
			u(x,t)\leq z(x) + K e^{-\mu^s Ht},\quad x\in\overline{\Omega},\quad t>0.
		\end{equation}
		with $u(x,t)$ the unique viscosity solution of \eqref{ev_eq} given 
		by Theorem \ref{teo-1.intro}, $z$ is the unique viscosity solution of \eqref{eliptico0}
		and 
		\begin{equation}\label{laH}
			0<H\coloneqq \int_{\mathbb{R}}\dfrac{1-e^{-\tau^2}}{|\tau|^{1+2s}} d\tau<\infty.
		\end{equation}
		
	\end{proposition}
	}
	\begin{proof}
	Notice that since $g$ is an affine function {and $s>1/2$}, we have that $\Lambda_1^s g(x)=0$ for any $x\in\mathbb{R}^N.$
	        Then $v(x,t)=u(x,t)-g(x)$ is a solution of 
	        \[
                    \begin{cases}
			v_t(x,t)=\Lambda_1^s v(x,t) &\text{in }\Omega\times(0,T),\\
			v(x,t)=0 &\text{in }
				(\mathbb{R}^N\setminus\Omega)\times\textcolor{black}{[}0,T),\\
			v(x,0)=u_0(x)-g(x) &\text{in }\textcolor{black}{\Omega}.
	    	    \end{cases}	        
	        \]
		Thus we could assume that $g\equiv0$ that in particular implies $z\equiv 0$. 
		Let us now consider that $\Omega\subset\subset B_{R}\coloneqq B_R(0)$, for some $R>0$, 
		and 
		\[
			0< w(x,t)\coloneqq C e^{\mu^s(R^2-Ht)}e^{-\mu|x|^2},\quad x\in \mathbb{R}^N,\, 
			t>0,\, C>0,
		\]
		where $H$ is given by \eqref{laH} and $\mu>0$ is a fixed but arbitrary parameter. 
		If for any $x\in\mathbb{R}^N,$ we choose $\theta_{x}\in\mathbb{S}^{N-1}$
		such that $\langle x, \theta_{x} \rangle = 0$ (if $x=0$
		we just choose any $\theta_{x} \in\mathbb{S}^{N-1}$), we have that
		\[
		\begin{array}{l}
		\displaystyle 
		    \Lambda_1^s w(x,t) 
	 \leq \int_{\mathbb{R}}\dfrac{w(x+\tau\theta_{x},t)-w(x,t)}{|\tau|^{1+2s}}d\tau 
		 = w(x,t) \int_{\mathbb{R}}\dfrac{e^{- \mu |\tau|^2}- 1}{|\tau|^{1+2s}}d\tau\\[10pt]
		 \displaystyle 
		  = - \mu^{s}H w(x,t)= w_t (x,t),
		\end{array}
		\]
		for every $\mu>0$.
		Moreover, since $\Omega\subset\subset B_R$, by choosing $C \sim e^{\mu R^2}$ big enough, we also obtain 
		$$
			w(x,0)=C e^{ \mu^s R^2-\mu|x|^2}\geq u_0(x),\quad x\in\textcolor{black}{\Omega.}
		$$  
		That is, $w(x,t)$ is a supersolution of \eqref{ev_eq} so by Theorem \ref{comparacion}, by taking, for instance, $K=C e^{ \mu^s R^2}>0$, the conclusion follows.
	\end{proof}

	In the special case that $g\equiv 0$ and $u_0$ compactly supported 
	we have that the lower exponential rate of Theorem \ref{asymp} is given by the first eigenvalue in $\Omega$.
	Notice that in this case 
	$z \equiv 0$ is the unique viscosity solution of \eqref{eliptico0}.
	
	\begin{proposition}\label{particular.fin}
	    Let $s>\nicefrac12.$
		If $g\equiv 0$ and $u_0$ is compactly supported inside
		$\Omega$ then,  for every $\mu>0$ there
		exists a constant $K=K(\mu)$ such that
		\[
		- C e^{-\mu_1 t}	\leq u(x,t)\leq  K e^{-\mu^s Ht},\quad x\in\overline{\Omega},
		\quad t>0.
		\]
		with $u(x,t)$ the unique viscosity solution of \eqref{ev_eq} given 
		by Theorem \ref{teo-1.intro}, $H$ is the universal constant defined in \eqref{laH} and $\mu_1>0$ is the first eigenvalue given in Theorem \ref{autof}. 
	\end{proposition}

	\begin{proof}
	    In the special case that $u_0$ is compactly supported we have that
		$$u_0 (x) \geq C \varphi_1(x), \qquad x \in \textcolor{black}{\Omega}, $$
		for a large constant $C$ (recall that the first eigenfunction is
		strictly negative in $\Omega$). Therefore, from a comparison argument, we get
		$u(x,t)\geq C \varphi_1(x)e^{-\mu_1 t}.$
		This bound together with \eqref{lonuestro} 
		give the result. 
	\end{proof}

	\bigskip
	
	\begin{remark}
	{\rm {\color{black} Notice that the two previous results imply that we have an 
	upper bound for $u(x,t) - z(x)$ that goes to zero as $t\to \infty$ 
	exponentially with an exponent as large as we want (by taking $\mu$ large).}
	
	Of course even when $g$ is an affine function, the upper bound obtained in Proposition \ref{particular} cannot be expected as a lower bound in general. That is, we could not expect, in general, that the solution of the evolution problem {\color{black}
	lies {above} the $s$-convex envelope plus a large exponential 
	in $\overline{\Omega}$. Indeed, for instance $u(x,t)=e^{-\mu_1 t}\varphi_1(x)<0\equiv z(x)$, where $(\mu_1,\varphi_1)$ were given in Theorem \ref{autof}, is a solution of \eqref{ev_eq} that it is decreasing at a fixed exponential speed.  } }
	\end{remark}

	 Moreover, when the datum is not an affine function we 
	have the following result.

	\begin{proposition}
		For every $x_0\in\Omega$ there exist $K_i>0$, $i=1,\,2$, 
		$u_0\in{\mathcal{C}(\mathbb{R}^N)}$ 
		and $g$ satisfy 
		\eqref{hip} such that  $u(x,t),$ the unique viscosity solution of \eqref{ev_eq}, satisfies
		$$
			u(x_0,t)\geq z(x_0)+K_1e^{-K_2t},\quad t>0,
		$$	
		with $z$ the unique viscosity solution of \eqref{eliptico0}.
	\end{proposition}

	\begin{proof}
	    Let us assume, without loss of generality that $x_0\in\Omega$ is the origin and  let us consider 
	    \begin{equation}\label{diff}
		B_r\coloneqq B_r(0)\subset\subset\Omega.
	    \end{equation} 
	    In what follows we  write $x=(x_1, x')\in\mathbb{R}^{N}$ and we define 
	    \[
	        B_r^{1}\coloneqq B_r\cap\{x'=0\}=[-r,r]\subseteq\mathbb{R}.
	    \]
		Therefore there exist (see \cite{SV}) $\mu_1^{1}>0$ a positive eigenvalue and $\varphi_1^{1}$ a 
		positive eigenfunction that is $\mathcal{C}^{s}$ up to the boundary such that
		$$	
			\begin{cases}
				-\Lambda_1^s\varphi_1^{1}(x)
				\left(=(-\Delta)^s_1 \varphi_1^{1}(x)\right)=\mu_{1}^1\varphi^{1}_1 &\text{in }[-r,r],\quad\\
					 \varphi_1^{1}=0&\text{in }(-\infty,-r)\cup(r,\infty).\\
			\end{cases}
		$$
		Since, by \eqref{diff}, if $(x_1,0)\in B_r$ clearly 
		$(x_1,0)\notin (\mathbb{R}^N\setminus\Omega)\cap\{x'=0\}$ 
		we can consider a nonnegative function 
		$g\in \mathcal{C}(\overline{\mathbb{R}^N\setminus\Omega})\cap 
		L^\infty(\overline{\mathbb{R} ^N\setminus\Omega})$ such that
		\begin{equation}\label{sobre_gg}
			g(x_1,x')\geq \varphi_1^{1}(x_1),\quad \{|x_1|\leq r,\, (x_1,x')\notin\Omega\}
		\end{equation}
		\begin{equation}\label{sobre_g}
			g(x_1,0)=0,\quad (x_1,0)\in (\mathbb{R}^{N}\setminus\Omega)\cap\{x'=0\},
		\end{equation}
		as a datum of \eqref{eliptico0}. On one hand, by the comparison principle 
		(see \cite[Theorem 2.2]{DPQR_FC}) it is follows that $z\geq 0$. 
		On the other hand, taking into account the geometric interpretation 
		of the $s$-convex envelope (see \cite[Theorem 1.1]{DPQR_FC}) $z\leq z_{D}$ 
		with $z_D$ the convex envelope of $g$ in $\Omega\cap\{x'=0\}$. Since, 
		by \eqref{sobre_g}, $z_D\equiv 0$ it follows that, in particular,
		\begin{equation}\label{zetaen0}
			z(0)=0. 
		\end{equation}

		Let us now take $u_0\geq 0$ such that 
		$\Lambda_1^su_0(x)\in L^{\infty}(\Omega)$ and
		\begin{equation}\label{sobre_u0}
			u_0(x_1,x')\geq \varphi^1_1(x_1),\quad |x_1|\leq r. 
		\end{equation}
		Thus, by \eqref{sobre_gg}, \eqref{sobre_u0} and the fact that 
		$\varphi_1^1(x_1)=0$ if $|x_1|>r$, the function 
		$$
			v(x_1, x', t)\coloneqq\varphi_1^1(x_1)e^{-\mu_1^1t}\left(\leq \varphi_1^1(x_1)\right),
			\quad t>0,$$ satisfies
		$$
		\begin{cases}
					v_t(x,t)=\Lambda_1^s v(x,t) &\text{in }\widetilde{\Omega}_r\times(0,\infty),\\
					{v(x_1,x',t)\leq g(x)} &\text{in }
						(\mathbb{R}^N\setminus\widetilde{\Omega}_r)\times(0,\infty),\\
					v(x_1,x',0)\leq u_0(x) &\text{in }\mathbb{R}^N,
				\end{cases}
		$$
		where
		$$
		    \widetilde{\Omega}_r\coloneqq\{ (x_1,x')\in\Omega\colon |x_1|\leq r\}.
		   $$
		Therefore by Theorem \ref{comparacion},
		we get that $u(x_1,x',t)\geq v(x_1, x', t)$ so, 
		by \eqref{zetaen0}, we conclude
		$$u(0,t)\geq z(0)+\varphi_1^{1}(0)e^{-\mu_1^1t},\quad t>0,$$
		as wanted.
\end{proof}

\section*{Acknowledgements}
BB was partially supported by AEI MTM2016-80474-P grant and Ram\'on y Cajal fellowship RYC2018-026098-I (Spain).

    LMDP and JDR are partially supported by
    Agencia Nacional de Promoción de la Investigación, el Desarrollo
    Tecnológico y la Innovación PICT-2018-3183, and
    PICT-2019-00985 and UBACYT 20020190100367 (Argentina).

LMDP was also partially
    supported by the European Union’s Horizon 2020 research and innovation program
    under the Marie Sklodowska-Curie grant agreement No 777822,

\end{document}